\documentclass{amsart}

\usepackage{setspace, amsmath, amsthm, amssymb, amsfonts, amscd, epic, graphicx, ulem, dsfont}
\usepackage[T1]{fontenc}
\usepackage{multirow}
\usepackage{bbm}

\makeatletter \@namedef{subjclassname@2010}{
  \textup{2010} Mathematics Subject Classification}
\makeatother

\newtheorem{lem}{Lemma}
\newtheorem{thm}{Theorem}
\newtheorem{pro}{Proposition}

\newtheorem{cor}{Corollary}

\theoremstyle{remark}
\newtheorem*{rema}{Remark}

\theoremstyle{definition}

\newtheorem{exa}{\textbf{Example}}

\newcommand{\Ima}{\text{\textrm{Im}}}

\newcommand{\C}{\mathbb{C}}

\begin{document}

\title{Exponentials of Bounded Normal Operators}
\author[A. CHABAN AND M. H. MORTAD]{AICHA CHABAN $^1$ AND MOHAMMED HICHEM MORTAD $^2$*}

\dedicatory{}
\thanks{* Corresponding author}
\date{}


\address{$^1$ Department of Mathematics, University of Chlef (Hassiba
Benbouali). BP 151, Chlef, 20000. Algeria.}
\email{aichachaban@yahoo.fr} \address{$^2$ Department of
Mathematics, University of Oran, B.P. 1524, El Menouar, Oran 31000,
Algeria.\newline {\bf Mailing address for the corresponding author}:
\newline Dr Mohammed Hichem Mortad \newline BP 7085 Seddikia Oran
\newline 31013 \newline Algeria}

\email{mhmortad@gmail.com, mortad@univ-oran.dz.}

\begin{abstract}
The present paper is mainly concerned with equations involving
exponentials of bounded normal operators. Conditions implying
commutativity of those normal operators are given. This is carried
out without the known $2\pi i$-congruence-free hypothesis. It is
also a continuation of a recent work by the corresponding author.
\end{abstract}

\maketitle

\section{Introduction}
First, we assume the reader is familiar with notions and results on
Bounded Operator Theory. Some important references are \cite{Con}
and \cite{RUD}. We suppose that all operators are linear and defined
on a complex Hilbert space, designated by $H$. The set of all these
operators is denoted by $B(H)$ which is a Banach algebra.

Let us just say a few words about notations. It is known that any
linear bounded operator $T$ may be expressed as $A+iB$ where $A$ and
$B$ are self-adjoint. In fact, it is known that $A=\frac{T+T^*}{2}$
and $B=\frac{T-T^*}{2i}$. Then we call $A$ the real part of $T$ and
we write Re$T$;  $B$ the imaginary part of $T$ and we write Im$T$.
It is also well-known that $T$ is normal if and only if $AB=BA$.

The following standard result will be useful.

\begin{lem}\label{e^A=I A s.a. A=0 Lemma}
Let $T$ be a self-adjoint operator such that $e^T=I$. Then $T=0$.
\end{lem}

We include a proof for the reader's convenience.

\begin{proof}
Assume $e^T=I$ and let $x\in\sigma(T)$. Then $e^x=1$. Since $T$ is
self-adjoint, $x$ is real and hence $e^x=1$ implies $x=0$ only.
Since $\sigma(T)$ is never empty, $\sigma(T)=\{0\}$. Again, since
$T$ is self-adjoint, by the "spectral radius theorem", we have
\[\|T\|=r(T)=0\Longrightarrow T=0.\]
\end{proof}

We will also be using the celebrated Fuglede theorem, which we
recall for the reader's convenience (for more versions, see
\cite{MHM1,MHM-Fug-Put-GMJ,Mortad-Fuglede-Putnam-CAOT-2011}). For a
proof, see e.g. \cite{Con} or \cite{RUD}.

\begin{thm}[Fuglede]
Let $A,N\in B(H)$. Assume that $N$ is normal. Then
\[AN=NA\Longrightarrow AN^*=N^*A\]
\end{thm}

The exponential of an operator appears in many areas of science, and
in mathematics in particular. For example, it appears when solving a
problem of the type $X'=AX$ where $A$ is an operator. It is also met
when dealing with semi-groups, the Stone theorem, the Lie product
formula, the Trotter product formula, the Feynman-Kac formula,
(bounded) wave operators, etc... See
\cite{Kat-book,Mortad-Scattering,RS1,RS2,RS3}.

Commutativity of operators and its characterization is one of the
most important topics in Operator Theory. Thus when the
commutativity of exponentials implies that of operators becomes an
interesting problem. In this paper we are mainly concerned with
problems of this sort. Many authors have worked previously on
similar questions (see
\cite{Fotios-Plaiogannis,schmoeger1999,schmoeger2000,schmoeger2001,wermuth.1997.pams}).
However, they all used what is known as the $2\pi i$-congruence-free
hypothesis (and similar hypotheses). G. Bourgeois (\cite{Bourgeois})
dropped that hypothesis but he only worked in low dimensions. Very
recently, M. H. Mortad (see \cite{Mortad-Exp-normal-Colloquium})
gave a different approach to this problem for normal operators, by
bypassing the $2\pi i$-congruence-free hypothesis. He used the
well-known cartesian decomposition of normal operators as $A+iB$
where $A$ and $B$ are commuting self-adjoint operators, so that the
following result may be applied (whose proof may be found in e.g.
\cite{wermuth.1997.pams})

\begin{thm}\label{e^Ae^B=e^Be^A iff A=B A,B self-adjoint}
Let $A$ and $B$ be two self-adjoint operators defined on a Hilbert
space. Then
\[e^Ae^B=e^Be^A \Longleftrightarrow AB=BA.\]
\end{thm}

Using a result on similarities (due to S. K. Berberian in
\cite{Berberian-adjoint-similar}), the following two results were
then obtained (both appeared in \cite{Mortad-Exp-normal-Colloquium})

\begin{pro}\label{e^Se^N S self-adjoint, N normal}
Let $N$ be a normal operator with cartesian decomposition $A+iB$.
Let $S$ be a self-adjoint operator.  If $\sigma(B)\subset (0,\pi)$,
then
\[e^Se^N=e^Ne^S\Longleftrightarrow SN=NS.\]
\end{pro}

\begin{rema}
Going back to the proof of Proposition \ref{e^Se^N S self-adjoint, N
normal}, we see that we may take $(-\frac{\pi}{2},\frac{\pi}{2})$ in
lieu of $(0,\pi)$ without any problem. Hence the same results hold
with this new interval. Thus any self-adjoint operator (remember
that its imaginary part then must vanish) obeys the given condition
on the spectrum.
\end{rema}

\begin{thm}\label{e^Me^N N,M normal operators}
Let $N$ and $M$ be two normal operators with cartesian
decompositions $A+iB$ and $C+iD$ respectively. If
$\sigma(B),\sigma(D)\subset (0,\pi)$, then
\[e^Me^N=e^Ne^M\Longleftrightarrow MN=NM.\]
\end{thm}

In this paper, we investigate further this question and accordingly,
we obtain more results. We have tried to keep the proofs as simple
as possible, so that we would allow a broader audience to read the
paper with ease. Another virtue of these proofs, is that in some
cases, they may even be applied to prove known results which use the
$2\pi i$-congruence-free hypothesis, for example. Let us now give a
 \textit{sample} of already known results on the topic of the present
paper.

\begin{thm}[Hille, \cite{Hille-1958}]Let $A$ and $B$ be both in $B(H)$ such that $e^A=e^B$. If
$\sigma(A)$ is incongruent (mod $2\pi i$), then $A$ and $B$ commute.
\end{thm}

\begin{thm}[Schmoeger, \cite{schmoeger2001}]
Let $A$ and $B$ be both in $B(H)$. Then
\begin{enumerate}
  \item If $A+B$ is normal, $\sigma(A+B)$ is generalized $2\pi
  i$-congruence-free and
  \[e^Ae^B=e^Be^A=e^{A+B},\]
  then $AB=BA$.
  \item If $A$ is normal, $\sigma(A)$ is generalized $2\pi
  i$-congruence-free and
  \[e^A=e^B,\]
  then $AB=BA$.
\end{enumerate}
\end{thm}

\begin{thm}[Schmoeger, \cite{schmoeger2003}]\label{Schmoeger 2003}Let $A$ and $B$ be both in $B(H)$ such that
$e^A=e^B$. Assume that $A$ is normal.
\begin{enumerate}
  \item If $r(A)<\pi$, then $AB=BA$ (where $r(A)$ is the spectral
  radius of $A$).
  \item If $\sigma(A)$ satisfies
  \[\sigma(A)\subseteq \{z\in\C:~|\text{Ima}~z|\leq \pi\}\]
  and
  \[\sigma(A)\cap \sigma(A+2\pi i)\subseteq \{i\pi\},\]
  then $A^2B=BA^2$. If $i\pi\not\in \sigma_p(A)$ or $-i\pi\not\in
  \sigma_p(A)$, then $AB=BA$.
\end{enumerate}
\end{thm}

Throughout this paper, the reader will see that with simple
hypotheses, we shall get the same conclusions as above.

\section{An Example}
The following example (s) will recalled on in the next section,
mainly as a counterexample.

\begin{exa}\label{Main Example!!!}
Let
\[A=\left(
      \begin{array}{cc}
        0 & \pi \\
        -\pi & 0 \\
      \end{array}
    \right).
\]
Then $A$ is clearly normal. By computing integers powers of $A$ we
may easily check that
\[e^A=\left(
        \begin{array}{cc}
          \cos\pi & \sin\pi \\
          -\sin\pi & \cos\pi \\
        \end{array}
      \right)=
\left(
                                                        \begin{array}{cc}
                                                          -1 & 0 \\
                                                          0 & -1 \\
                                                        \end{array}
                                                      \right)=-I.\]
Next, we have
\[\text{Im}A=\frac{A-A^*}{2i}=\left(
      \begin{array}{cc}
        0 & -i\pi \\
        i\pi & 0 \\
      \end{array}
    \right)\]
    and hence $\sigma(\Ima A)=\{\pi,-\pi\}$. This signifies that
    $\sigma(\Ima A)$ cannot be inside an open interval of length
    equals to $\pi$ (a hypothesis which will play an important role in our proofs).

Now let
\[B=\left(
      \begin{array}{cc}
        \pi & -2\pi \\
        \pi & -\pi \\
      \end{array}
    \right).
\]
We may also show that $e^B=-I$. Finally, it is easily verifiable
that \textit{$A$ and $B$ do not commute}.
\end{exa}

\section{Main Results}

We start by giving a very important result of the paper. We have
\begin{thm}\label{e^AT=Te^A commuting!!!!}
Let $A$ be in $B(H)$. Let $N\in B(H)$ be normal such that
$\sigma(\Ima N)\subset (0,\pi)$. Then
\[Ae^N=e^NA\Longleftrightarrow AN=NA.\]
\end{thm}

\begin{proof}
Of course, we are only concerned with proving the implication
"$\Longrightarrow$". The normality of $N$ implies that of $e^N$, and
so by the Fuglede theorem
\[Ae^N=e^NA\Longleftrightarrow Ae^{N^*}=e^{N^*}A\]
or
\[A^*e^N=e^NA^*.\]
Hence
\[(A+A^*)e^N=e^N(A+A^*) \text{ or (Re}A) e^N=e^N(\text{Re}A)\]
so that
\[e^{\text{Re}A}e^N=e^Ne^{\text{Re}A}.\]
But, \text{Re}$A$ is self-adjoint, so Proposition \ref{e^Se^N S
self-adjoint, N normal} applies and then gives
\[(\text{Re}A)N=N(\text{Re}A).\]
Similarly, we find that
\[\text{(Im}A) e^N=e^N(\text{Im}A)\]
and as $\text{Im}A$ is self-adjoint, similar arguments to those
applied before yield
\[(\text{Im}A)N=N(\text{Im}A).\]
Therefore, $AN=NA$, establishing the result.
\end{proof}

\begin{rema}
The hypothesis $\sigma(\text{Im}N)\subset (0,\pi)$ cannot merely be
dropped. Take $N$ to be the operator $A$ in Example \ref{Main
Example!!!}, and take $A$ to be any operator which does not commute
with $N$. Then
\[Ae^N=e^NA=-A\text{ but } AN\neq NA.\]
\end{rema}

Next we give the first consequence of the previous result. We have

\begin{thm}\label{eA+B A+B normal!!!!!}
Let $A$ and $B$ be  both in $B(H)$. Assume that $A+B$ is normal such
that $\sigma(\text{Im}(A+B))\subset (0,\pi)$. If
\[e^Ae^B=e^Be^A=e^{A+B},\]
then $AB=BA$.
\end{thm}

\begin{proof}
We have
\[e^{A+B}e^A=e^Be^Ae^A=e^Ae^Be^A=e^Ae^{A+B}.\]
Since $A+B$ is normal and $\sigma(\text{Im}(A+B))\subset (0,\pi)$,
Theorem \ref{e^AT=Te^A commuting!!!!} gives
\[(A+B)e^A=e^A(A+B) \text{ or just } Be^A=e^AB\]
for $A$ commutes with $e^A$. Now, right multiplying both sides of
the previous equation by $e^B$ leads to
\[Be^Ae^B=e^ABe^B=e^Ae^BB\]
or
\[Be^{A+B}=e^{A+B}B.\]
Applying again Theorem \ref{e^AT=Te^A commuting!!!!}, we see that
\[B(A+B)=(A+B)B \text{ or } AB=BA.\]
The proof is thus complete.
\end{proof}

\begin{cor}
Let $A\in B(H)$. Then
\[e^Ae^{A^*}=e^{A^*}e^A=e^{A+A^*}\Longleftrightarrow A\text{ is normal.}\]
\end{cor}

\begin{proof}
We need only prove the implication "$\Longrightarrow$". It is plain
that $A+A^*$ is self-adjoint. Hence the remark below Proposition
\ref{e^Se^N S self-adjoint, N normal} combined with Theorem
\ref{eA+B A+B normal!!!!!} give us
\[AA^*=A^*A.\]
\end{proof}

We have yet another consequence of Theorem \ref{e^AT=Te^A
commuting!!!!} (cf. Theorem \ref{Schmoeger 2003}).

\begin{cor}\label{e^A=e^B A^2B=BA^2 A normal imaginary spectrum cramped!!!}
Let $A$ be normal such that $\sigma(\text{Im}A)\subset (0,\pi)$. Let
$B\in B(H)$. Then
\[e^A=e^B\Longrightarrow A^2B=BA^2.\]
\end{cor}

\begin{proof}We obviously have
\[e^B(e^BB)=(Be^B)e^B.\]
So since $e^A=e^B$, we have
\[e^A(e^AB)=(Be^A)e^A.\]
By Theorem \ref{e^AT=Te^A commuting!!!!}, we obtain
\[Ae^AB=Be^AA.\]
Hence
\[e^A(AB)=(BA)e^A.\]
Applying Theorem \ref{e^AT=Te^A commuting!!!!} once more yields
\[A(AB)=(BA)A \text{ or } A^2B=BA^2,\]
completing the proof.
\end{proof}

\begin{cor}\label{e^A=e^B AB=BA A normal imaginary spectrum cramped!!!}
Let $A$ be normal such that $\sigma(\text{Im}A)\subset (0,\pi)$. Let
$B\in B(H)$. Then
\[e^A=e^B\Longrightarrow AB=BA.\]
\end{cor}

\begin{rema}
By Example \ref{Main Example!!!}, $e^{A}=e^{B}~(=-I)$, but $AB\neq
BA$, showing again the importance of the assumption
$\sigma(\text{Im}A)\subset (0,\pi)$.
\end{rema}

\begin{proof}
We obviously have
\[Be^B=e^BB\]
so that
\[Be^A=e^AB.\]
Theorem \ref{e^AT=Te^A commuting!!!!} does the remaining job, i.e.
it gives the commutativity of $A$ and $B$.
\end{proof}

\begin{cor}\label{e^A=e^B A, B self-adjoint A=B}
Let $A$ and $B$ be two self-adjoint operators. Then
\[e^A=e^B\Longleftrightarrow A=B.\]
\end{cor}

\begin{proof}
Observe first that we are only concerned with establishing the
implication "$\Longrightarrow$". By the remark below Proposition
\ref{e^Se^N S self-adjoint, N normal}, we may get
\[e^A=e^B\Longrightarrow e^Ae^B=e^Be^A\Longrightarrow AB=BA.\]
Hence
\[I=e^Ae^{-A}=e^{A}e^{-B}=e^{A-B}\]
 since $A$ and $B$ commute. But $A-B$ is obviously
self-adjoint, so Lemma \ref{e^A=I A s.a. A=0 Lemma} gives $A=B$.
\end{proof}

\begin{rema}
The previous corollary is actually a consequence of Theorem
\ref{e^Ae^B=e^Be^A iff A=B A,B self-adjoint}. Authors who did not
argue as we did usually considered it as a consequence of their
results. But, with the proof given here, we clearly see that we only
need Theorem \ref{e^Ae^B=e^Be^A iff A=B A,B self-adjoint} and Lemma
\ref{e^A=I A s.a. A=0 Lemma}.
\end{rema}

\begin{rema}
Of course, the previous corollary also generalizes Lemma \ref{e^A=I
A s.a. A=0 Lemma}.
\end{rema}

\begin{cor}
Let $A$ be normal. Then we have
\[A \text{ is self-adjoint }\Longleftrightarrow e^{iA} \text{ is unitary.}\]
\end{cor}

\begin{proof}
The implication "$\Longrightarrow$" is well-known. Let us prove the
reverse implication. By the normality of $A$, we have
\[e^{iA-iA^*}=e^{iA}e^{-iA^*}=e^{iA}(e^{iA})^*=I.\]
Since $iA-iA^*$ is self-adjoint, Lemma \ref{e^A=I A s.a. A=0 Lemma}
gives $A=A^*$, which completes the proof.
\end{proof}

We now come to a result that appeared in \cite{Fotios-Plaiogannis}
and \cite{schmoeger2001}. It reads: If $A$ is self-adjoint,
$\sigma(A)\subseteq [-\pi,\pi]$ and $e^{iA}=e^B$, then $B^*=-B$ if
$B$ is normal. Here is an improvement of this result.

\begin{pro}
If $A$ is self-adjoint and $e^{iA}=e^B$, then $B^*=-B$ whenever $B$
is normal.
\end{pro}

\begin{proof}
It is clear that $e^{iA}$ is unitary. We also have
\[e^{B^{*}}=e^{-iA} \text{ and } e^{-B}=e^{-iA}.\]
Thus
\[e^{-B}=e^{B^*} \text{ so that } e^{B+B^*}=I\]
because $B$ is normal. However, $B+B^*$ is always self-adjoint,
whence $B^*=-B$ by Lemma \ref{e^A=I A s.a. A=0 Lemma}.
\end{proof}

\section{Conclusion}

The results of this paper as well as those of the paper
\cite{Mortad-Exp-normal-Colloquium} should be easily generalized to
unital Banach algebras.

Theorem \ref{e^AT=Te^A commuting!!!!} is very important here. The
very simple and interesting proof of Corollary \ref{e^A=e^B
A^2B=BA^2 A normal imaginary spectrum cramped!!!} or Corollary
\ref{e^A=e^B AB=BA A normal imaginary spectrum cramped!!!} could not
have been achieved if we had not Theorem \ref{e^AT=Te^A
commuting!!!!} in hand. Also, as mentioned in the introduction, most
of the proofs, for instance that of Corollary \ref{e^A=e^B A, B
self-adjoint A=B}, may be adopted to prove the results that use the
$2\pi i$-congruence-free hypothesis and similar hypotheses.

\bibliographystyle{amsplain}

\end{document}